\documentclass[reqno,11pt]{article} 
\usepackage[left=2.5cm,right=2.5cm,top=3cm,bottom=3cm,a4paper]{geometry}
\usepackage{amsmath, amssymb}
\usepackage{amsthm, amscd} 
\usepackage{empheq}
\usepackage[all,cmtip]{xy}
\usepackage{float}
\usepackage{caption}
\usepackage{color}
\usepackage{rotating}
\usepackage[utf8]{inputenc}
\usepackage[T1]{fontenc}

\makeatletter
\DeclareRobustCommand{\rvdots}{%
  \vbox{
    \baselineskip4\p@\lineskiplimit\z@
    \kern-\p@
    \hbox{.}\hbox{.}\hbox{.}
  }}
\makeatother

\newcommand{\Mod}[1]{\ (\textup{mod}\ #1)}
\makeatletter
\def\moverlay{\mathpalette\mov@rlay}
\def\mov@rlay#1#2{\leavevmode\vtop{%
   \baselineskip\z@skip \lineskiplimit-\maxdimen
   \ialign{\hfil$\m@th#1##$\hfil\cr#2\crcr}}}
\newcommand{\charfusion}[3][\mathord]{
    #1{\ifx#1\mathop\vphantom{#2}\fi
        \mathpalette\mov@rlay{#2\cr#3}
      }
    \ifx#1\mathop\expandafter\displaylimits\fi}
\makeatother

\theoremstyle{plain} 
\newtheorem{theorem}{\indent\sc Theorem}[section]
\newtheorem{lemma}[theorem]{\indent\sc Lemma}
\newtheorem{corollary}[theorem]{\indent\sc Corollary}
\newtheorem{proposition}[theorem]{\indent\sc Proposition}

\theoremstyle{definition} 
\newtheorem{definition}[theorem]{\indent\sc Definition}

\newtheorem{remark}[theorem]{\indent\sc Remark}

\newtheorem{thmx}{Theorem}

\makeatletter
\def\address#1#2{\begingroup
\noindent\parbox[t]{7.8cm}{%
\small{\scshape\ignorespaces#1}\par\vskip1ex
\noindent\small{\itshape E-mail address}%
\/: #2\par\vskip4ex}\hfill%
\endgroup}%
\makeatother

\title{Inverse limits of CM points on certain Shimura varieties}
\author{
\textsc{Ho Yun Jung, Ja Kyung Koo and Dong Hwa Shin} 
}
\date{} 
\begin{document}

\allowdisplaybreaks

\maketitle

 \footnote{
\thanks{}
}

\begin{abstract}
Let $N$ be a positive integer, and let $D\equiv0$ or $1\Mod{4}$ be a negative integer. 
We define the sets $\mathcal{CM}(D,\,Y_1(N)^\pm)$ and $\mathcal{CM}(D,\,Y(N)^\pm)$ 
as subsets of the Shimura varieties $Y_1(N)^\pm$ and $Y(N)^\pm$, respectively, 
consisting of CM points of discriminant $D$ that are primitive modulo $N$.
By using the theory of definite form class groups, 
we show that the inverse limits
\begin{equation*}
\varprojlim_N\,\mathcal{CM}(D,\,Y_1(N)^\pm)\quad\textrm{and}\quad
\varprojlim_N\,\mathcal{CM}(D,\,Y(N)^\pm)
\end{equation*}
naturally inherit group structures
isomorphic to $\mathrm{Gal}(K^\mathrm{ab}/\mathbb{Q})$ and
$\mathrm{Gal}(K^\mathrm{ab}(t^{1/\infty})/\mathbb{Q}(t))$, respectively,
where $K=\mathbb{Q}(\sqrt{D})$ and $t$ is a transcendental number. 
These results provide an explicit and geometric interpretation of class field theory in terms of
inverse limits of CM points on the associated Shimura varieties.
\end{abstract}

\tableofcontents

\section{Introduction}

Let $\mathbb{H}=\{\tau\in\mathbb{C}~|~\mathrm{Im}(\tau)>0\}$ be the complex upper half-plane.
A congruence subgroup $\Gamma$ of $\mathrm{SL}_2(\mathbb{Z})$,
acting on $\mathbb{H}$ via fractional linear transformations,
defines the modular curve $\Gamma\backslash\mathbb{H} =\{[\tau]:=\Gamma\tau~|~\tau\in\mathbb{H}\}$. 
In particular, for each positive integer $N$, we consider the modular curves 
\begin{equation*}
Y(N)=\Gamma(N)\backslash\mathbb{H}
\quad
\textrm{and}\quad
Y_1(N)=\Gamma_1(N)\backslash\mathbb{H}
\end{equation*}
which parametrize isomorphism classes of complex elliptic curves equipped,
respectively, with
full level $N$ structure and with a point of order $N$
(see \cite[Theorem 1.5.1]{D-S}). Here, 
$\Gamma(N)$ denotes the principal congruence subgroup of level $N$, and
\begin{equation*}
\Gamma_1(N)=\left\{
\gamma\in\mathrm{SL}_2(\mathbb{Z})~\bigg|~\gamma\equiv\begin{bmatrix}1&\mathrm{*}\\0&1\end{bmatrix}
\Mod{N\cdot M_2(\mathbb{Z})}
\right\}.
\end{equation*}
There exists a natural forgetful map
$Y(N)\twoheadrightarrow Y_1(N)$.
These modular curves can be viewed as Riemann surfaces, 
which compactify into smooth projective algebraic curves $X(N)$ and
$X_1(N)$, respectively. 
It is well known that $X(N)$ and $X_1(N)$
are defined over $\mathbb{Q}(\mu_N)$ and $\mathbb{Q}$, respectively, 
where $\mu_N$ denotes the group of $N$th roots of unity (see \cite[Chapters 2 and 7]{D-S}). 
\par
Let $(\mathbb{N},\,\preccurlyeq)$ be the directed partially ordered set of 
positive integers ordered by divisibility\,: for $n,\,m\in\mathbb{N}$, 
we write $n\preccurlyeq m$ if and only if $n\,|\,m$. 
In their seminal work \cite{D-Z}, Daw and Zilber used the theory of canonical models of Shimura varieties to 
describe the projective limit $\varprojlim_{N\in\mathbb{N}} Y(N)$ and its automorphism group. 
A key result of their work is the realization of the Galois group 
$\mathrm{Gal}(\mathbb{Q}(\mathrm{CM})/\mathbb{Q})$
as an open subgroup of the automorphism group of a certain adelic structure,
where $\mathrm{CM}$ denotes the union of all CM points on the modular curves $Y(N)$. 
\par
In contrast, our approach is more explicit. 
Let $\mathbb{H}^-=\{\tau\in\mathbb{C}~|~\mathrm{Im}(\tau)<0\}$ be the lower half-plane.
Rather than considering the full inverse system 
of modular curves $Y(N)$, we restrict our attention to the 
inverse systems consisting of certain CM points on Shimura varieties
\begin{equation*}
Y_1(N)^\pm:=Y_1(N)\sqcup Y_1(N)^-
\quad\textrm{and}\quad
Y(N)^\pm:=Y(N)\sqcup Y(N)^-,
\end{equation*}
where  $Y_1(N)^-:=\Gamma_1(N)\backslash\mathbb{H}^-$ and
$Y(N)^-:=\Gamma(N)\backslash\mathbb{H}^-$.
Define
\begin{equation*}
\widetilde{\mathbb{H}}_1^\pm:=\varprojlim_{N\in\mathbb{N}}Y_1(N)^\pm
\quad\textrm{and}\quad
\widetilde{\mathbb{H}}^\pm:=\varprojlim_{N\in\mathbb{N}}Y(N)^\pm.
\end{equation*}

\begin{thmx}[Theorem \ref{main}]\label{ThmA}
Let $D\equiv0$ or $1\Mod{4}$ be a negative integer, 
and let $K=\mathbb{Q}(\sqrt{D})$ be the 
corresponding imaginary quadratic field. 
Let $t$ be a transcendental number. 
\begin{enumerate}
\item[\textup{(i)}] The subset $\mathcal{CM}(D,\,\widetilde{\mathbb{H}}_1^\pm)$ 
of $\widetilde{\mathbb{H}}_1^\pm$ admits a group structure isomorphic to
$\mathrm{Gal}(K^\mathrm{ab}/\mathbb{Q})$.
\item[\textup{(ii)}] 
Define
\begin{equation*}
K^\mathrm{ab}(t^{1/\infty}):=\bigcup_{N\geq1}K^\mathrm{ab}(t^{1/N}). 
\end{equation*}
The subset
$\mathcal{CM}(D,\,\widetilde{\mathbb{H}}^\pm)$ of 
$\widetilde{\mathbb{H}}^\pm$ can be regarded as a group isomorphic to 
$\mathrm{Gal}(K^\mathrm{ab}(t^{1/\infty})/\mathbb{Q}(t))$.
Moreover, $\mathcal{CM}(D,\,\widetilde{\mathbb{H}}_1^\pm)$ 
can be identified with a quotient group
of $\mathcal{CM}(D,\,\widetilde{\mathbb{H}}^\pm)$
 via the natural forgetful maps. 
\end{enumerate}
\end{thmx}

For the precise definitions of $\mathcal{CM}(D,\,\widetilde{\mathbb{H}}_1^\pm)$  and
$\mathcal{CM}(D,\,\widetilde{\mathbb{H}}^\pm)$ in Theorem \ref{ThmA}, see Definitions \ref{CMDN} and
\ref{CMDH}. 

\section {Definite form class groups}

We shall present an explicit approach to class field theory 
via definite form class groups, as recently developed by 
Jung et al. in \cite{J-K-S-Y}.
\par
Throughout this paper, let $N$ be a positive integer and
$D$ be a negative integer such that $D\equiv0$ or $1\Mod{4}$.
Define the set $\mathcal{Q}(D,\,N)$ of primitive positive definite binary quadratic forms
of discriminant $D$, subject to an additional condition, as follows\,:
\begin{equation*}
\mathcal{Q}(D,\,N):=\left\{Q\left(\begin{bmatrix}x\\y\end{bmatrix}\right)=
ax^2+bxy+cy^2\in\mathbb{Z}[x,\,y]~\Bigg|~\begin{array}{l}
\gcd(a,\,b,\,c)=1,\\
b^2-4ac=D,~a>0,\\
\gcd(a,\,N)=1\end{array}\right\}.
\end{equation*}
For each $Q(x,\,y)=ax^2+bxy+cy^2\in\mathcal{Q}(D,\,N)$, let $\omega_Q\in\mathbb{H}$ denote the root of the
quadratic polynomial $Q(x,\,1)$, that is,
\begin{equation*}
\omega_Q=\frac{-b+\sqrt{D}}{2a}. 
\end{equation*}
We extend this set by defining
\begin{equation*}
\mathcal{Q}(D,\,N)^\pm:
=\mathcal{Q}(D,\,N)\sqcup\mathcal{Q}(D,\,N)^-,~~
\textrm{where}~
\mathcal{Q}(D,\,N)^-:=\{-Q~|~Q\in\mathcal{Q}(D,\,N)\}.
\end{equation*}
Let $\Gamma$ be either $\Gamma_1(N)$ or $\Gamma(N)$. 
Then $\Gamma$ acts on $\mathcal{Q}(D,\,N)^\pm$ from the right by
\begin{equation*}
Q\left(\begin{bmatrix}x\\y\end{bmatrix}\right)^\gamma=
Q\left(\gamma\begin{bmatrix}x\\y\end{bmatrix}\right)\quad
(Q\in\mathcal{Q}(D,\,N)^\pm,~\gamma\in\Gamma),
\end{equation*}
which induces an equivalence relation $\sim_\Gamma$ defined by\,:
for $Q,\,Q'\in\mathcal{Q}(D,\,N)^\pm$, 
\begin{equation*}
Q\sim_\Gamma Q'\quad\Longleftrightarrow\quad
Q'=Q^\gamma~\textrm{for some}~\gamma\in\Gamma. 
\end{equation*}
Note that the sets $\mathcal{Q}(D,\,N)$ and 
$\mathcal{Q}(D,\,N)^-$
are invariant under this action. 
\par
Unless stated otherwise, let $K$ be the imaginary quadratic field given by $K=\mathbb{Q}(\sqrt{D})$ 
and $\mathcal{O}$ be the order in $K$ of discriminant $D$.  
Let $K_{\mathcal{O},\,N}$ denote the ray class field of $\mathcal{O}$ modulo $N\mathcal{O}$,
originally introduced by 
S\"{o}hngen in \cite{Sohngen} (cf. \cite[$\S$V.9]{Janusz} and \cite[$\S$2]{J-K-S-Y}).
More precisely, define the subgroups of proper fractional $\mathcal{O}$-ideals
\begin{align*}
I(\mathcal{O},\,N)&=\langle\textrm{nonzero proper $\mathcal{O}$-ideals prime to $N$}\rangle,\\
P_1(\mathcal{O},\,N)&=\langle\nu\mathcal{O}~|~\textrm{$\nu$ is a nonzero element of $\mathcal{O}$
such that}~\nu\equiv1\Mod{N\mathcal{O}}\rangle. 
\end{align*}
Let $\mathcal{O}_K$ be the maximal order of $K$, and let $\ell=[\mathcal{O}_K:\mathcal{O}]$ denote the conductor of $\mathcal{O}$. 
Then, by the existence theorem of class field theory, 
there exists a unique abelian extension $K_{\mathcal{O},\,N}$ of $K$ satisfying\,:
\begin{enumerate}
\item[(i)] Every nonzero prime ideal of $\mathcal{O}_K$ that ramifies in $K_{\mathcal{O},\,N}$ divides $\ell N\mathcal{O}_K$\,;
\item[(ii)] The Artin reciprocity map for the modulus $\ell N\mathcal{O}_K$ induces an isomorphism
\begin{equation}\label{IPG}
I(\mathcal{O},\,N)/P_1(\mathcal{O},\,N)\stackrel{\sim}{\rightarrow}
\mathrm{Gal}(K_{\mathcal{O},\,N}/K).
\end{equation}
\end{enumerate}

\begin{proposition}\label{groupstructure}
Let $t$ be a transcendental number. 
\begin{enumerate}
\item[\textup{(i)}] The set $\mathcal{Q}(D,\,N)/\sim_{\Gamma_1(N)}$ of equivalence classes admits a 
group structure via the well-defined bijection 
\begin{equation*}
\mathcal{Q}(D,\,N)/\sim_{\Gamma_1(N)}\,\rightarrow\,I(\mathcal{O},\,N)/P_1(\mathcal{O},\,N),
\quad[Q]\mapsto[\mathbb{Z}\omega_Q+\mathbb{Z}]. 
\end{equation*}
Hence $\mathcal{Q}(D,\,N)/\sim_{\Gamma_1(N)}$ is isomorphic to 
$\mathrm{Gal}(K_{\mathcal{O},\,N}/K)$ via the isomorphism \textup{(\ref{IPG})}. 
\item[\textup{(ii)}] There exists a well-defined bijection
\begin{equation*}
\mathcal{Q}(D,\,N)/\sim_{\Gamma(N)}\,\rightarrow\,\mathrm{Gal}(K_{\mathcal{O},\,N}(t^{1/N})/K(t))
\end{equation*}
which makes the following diagram \textup{(Figure \ref{diagram})} commute\,\textup{:}
\begin{figure}[H]
\begin{equation*}
\xymatrixcolsep{5pc}
\xymatrix{
\mathcal{Q}(D,\,N)/\sim_{\Gamma(N)} \ar@{->}[r]^{}
\ar@{->>}[dd]_{\textrm{the natural surjection}}
 &\mathrm{Gal}(K_{\mathcal{O},\,N}(t^{1/N})/K(t)) \ar@{->>}[dd]^{\textrm{the natural surjection}} \\\\
\mathcal{Q}(D,\,N)/\sim_{\Gamma_1(N)} \ar@{->}[r]^\sim & \mathrm{Gal}(K_{\mathcal{O},\,N}/K)
}
\end{equation*}
\caption{A diagram concerning Galois groups}
\label{diagram}
\end{figure}
\end{enumerate}
\end{proposition}
\begin{proof}
\begin{enumerate}
\item[(i)] See \cite[Proposition 9.3 and Theorem 9.4]{J-K-S-Y}. 
\item[(ii)] See \cite[Corollary 12.6]{J-K-S-Y}. 
\end{enumerate}
\end{proof}

The field  $K_{\mathcal{O},\,N}$ is further Galois over $\mathbb{Q}$, and we have
a semidirect product decomposition 
\begin{equation}\label{semidirect}
\mathrm{Gal}(K_{\mathcal{O},\,N}/\mathbb{Q})=
\mathrm{Gal}(K_{\mathcal{O},\,N}/K)\rtimes\langle\mathfrak{c}|_{K_{\mathcal{O},\,N}}\rangle
\end{equation}
where $\mathfrak{c}$ denotes the complex conjugation
(see \cite[Lemmas 6.1 and 7.1]{J-K-S-Y25}).  

\begin{corollary}\label{GalFormCor}
Let $t$ be a transcendental number. 
\begin{enumerate}
\item[\textup{(i)}] The set $\mathcal{Q}(D,\,N)^\pm/\sim_{\Gamma_1(N)}$
forms a group isomorphic to $\mathrm{Gal}(K_{\mathcal{O},\,N}/\mathbb{Q})$.
\item[\textup{(ii)}] The set
$\mathcal{Q}(D,\,N)^\pm/\sim_{\Gamma(N)}$
forms a group isomorphic to $\mathrm{Gal}(K_{\mathcal{O},\,N}(t^{1/N})/\mathbb{Q}(t))$.  
\end{enumerate}
\end{corollary}
\begin{proof}
This follows directly from Proposition \ref{groupstructure} and the decomposition (\ref{semidirect})\,;
see \cite[Proposition 7.2]{J-K-S-Y25}. 
\end{proof}

\begin{remark}\label{MNsurjective}
If $N,\,M\in\mathbb{N}$ with $N\preccurlyeq M$, then 
the natural maps
\begin{align*}
\mathcal{Q}(D,\,N)^\pm/\sim_{\Gamma_1(M)}&\rightarrow\mathcal{Q}(D,\,N)^\pm/\sim_{\Gamma_1(N)},\\
\mathcal{Q}(D,\,N)^\pm/\sim_{\Gamma(M)}&\rightarrow\mathcal{Q}(D,\,N)^\pm/\sim_{\Gamma(N)}
\end{align*}
are surjective, again by Proposition \ref{groupstructure}, the decomposition (\ref{semidirect}) and \cite[Proposition 7.2]{J-K-S-Y25}. 
\end{remark}

\section {Families of CM points on certain Shimura varieties}

Note that the disjoint union $Y(N)^\pm=Y(N)\sqcup Y(N)^-$ is isomorphic to the Shimura variety 
$\mathrm{Sh}(G,\,X)_{\mathcal{K}(N)}$
associated with the Shimura datum $(G,\,X)$, where
\begin{equation*}
G=\mathrm{SL}_2,\quad X=\mathbb{H}^\pm:=\mathbb{H}\sqcup
\mathbb{H}^-,\quad
\mathcal{K}(N)=\mathrm{ker}(\mathrm{SL}_2(\widehat{\mathbb{Z}})
\rightarrow\mathrm{SL}_2(\mathbb{Z}/N\mathbb{Z}))
\end{equation*}
(cf. \cite[$\S$2.2]{D-Z} or \cite[Definition 5.5]{Milne}). More precisely, 
if $\mathbb{A}_f$ denotes the ring 
of finite adeles of $\mathbb{Q}$ with restricted topology, then 
we have the isomorphism of complex manifolds
\begin{equation*}
\mathrm{Sh}(G,\,X)_{\mathcal{K}(N)}=G(\mathbb{Q})\backslash X\times
G(\mathbb{A}_f)/\mathcal{K}(N)\simeq Y(N)^\pm.  
\end{equation*}
Similarly, the space $Y_1(N)^\pm=
Y_1(N)\sqcup Y_1(N)^-$ can also be regarded as a Shimura variety 
associated with the same Shimura datum $(G,\,X)$. 
In this section, we shall construct two inverse systems
\begin{equation*}
\{\mathcal{CM}(D,\,Y_1(N)^\pm\}_{N\in\mathbb{N}}
\quad\textrm{and}\quad
\{\mathcal{CM}(D,\,Y(N)^\pm\}_{N\in\mathbb{N}},
\end{equation*}
where each object consists of certain CM points on
the Shimura varieties $Y_1(N)^\pm$ and $Y(N)^\pm$, respectively. 

\begin{definition}\label{CMDN}
Let $\tau\in\mathbb{H}^\pm$ be a CM point (i.e., $\tau$ is imaginary quadratic), and let
 $ax^2+bx+c\in\mathbb{Z}[x]$ be a primitive quadratic polynomial \textup{(}namely, $\gcd(a,\,b,\,c)=1$\textup{)}
 having $\tau$ as a root.  
\begin{enumerate}
\item[\textup{(i)}] The \textit{discriminant} $D_\tau$ of $\tau$ is defined to be the discriminant of
the polynomial $ax^2+bx+c$, namely, $D_\tau=b^2-4ac$. 
\item[\textup{(ii)}] We say that $\tau$ is \textit{primitive modulo $N$} if $\gcd(a,\,N)=1$.  
\item[\textup{(iii)}] Let $\mathbb{H}(D,\,N)^\pm$ denote the set of all  CM points $\tau\in\mathbb{H}^\pm$ 
such that $D_\tau=D$ and $\tau$ is primitive modulo $N$. 
\end{enumerate}
\end{definition}

\begin{definition}
Define the sets
\begin{align*}
\mathcal{CM}(D,\,Y_1(N)^\pm)&:=\{[\tau]\in Y_1(N)^\pm~|~\tau\in\mathbb{H}(D,\,N)^\pm\},\\
\mathcal{CM}(D,\,Y(N)^\pm)&:=\{[\tau]\in Y(N)^\pm~|~\tau\in\mathbb{H}(D,\,N)^\pm\}.
\end{align*}
\end{definition}

\begin{lemma}\label{formpoint}
Let $\Gamma$ be either $\Gamma_1(N)$ or $\Gamma(N)$. Then the map
\begin{equation}\label{rho}
\rho:\mathcal{Q}(D,\,N)^\pm/\sim_\Gamma\,\rightarrow\,
\mathcal{CM}(D,\,\Gamma\backslash\mathbb{H}^\pm),\quad[Q]\mapsto
\left\{
\begin{array}{ll}[\omega_Q] & \textrm{if}~Q\in\mathcal{Q}(D,\,N),\\
\mathrm{[}\overline{\omega}_{-Q}] & \textrm{if}~Q\in\mathcal{Q}(D,\,N)^-\end{array}\right.
\end{equation}
is a well-defined bijection. 
\end{lemma}
\begin{proof}
Let $Q,\,Q'\in\mathcal{Q}(D,\,N)^\pm$ such that $[Q]=[Q']$
in $\mathcal{Q}(D,\,N)^\pm/\sim_\Gamma$, and so
$Q'=Q^\gamma$ for some $\gamma\in\Gamma$. Then we see that $-Q'=-(Q^\gamma)=(-Q)^\gamma$ and 
\begin{equation*}
\left\{\begin{array}{ll}
\omega_{Q'}=\omega_{Q^\gamma}=\gamma^{-1}(\omega_Q) & \textrm{if}~Q\in\mathcal{Q}(D,\,N),\\
\overline{\omega}_{-Q'}=\overline{\omega}_{(-Q)^\gamma}
=\overline{\gamma^{-1}(\omega_{-Q}) }=
\gamma^{-1}(\overline{\omega}_{-Q})
& \textrm{if}~Q\in\mathcal{Q}(D,\,N)^-.
\end{array}\right.
\end{equation*}
This observation implies that $\rho$ is well defined.
\par
We deduce that for $Q,\,Q'\in\mathcal{Q}(D,\,N)^\pm$
\begin{align*}
\rho([Q])=\rho([Q'])&\Longleftrightarrow
\left\{
\begin{array}{ll}[\omega_Q]=[\omega_{Q'}] & \textrm{if}~Q\in\mathcal{Q}(D,\,N),\\
\mathrm{[}\overline{\omega}_{-Q}]=
[\overline{\omega}_{-Q'}] & \textrm{if}~Q\in\mathcal{Q}(D,\,N)^-\end{array}\right.\\
&\Longleftrightarrow\left\{
\begin{array}{ll}\omega_Q=\gamma(\omega_{Q'}) & \textrm{if}~Q\in\mathcal{Q}(D,\,N),\\
\overline{\omega}_{-Q}=\gamma(\overline{\omega}_{-Q'})
=\overline{\gamma(\omega_{-Q'})}
 & \textrm{if}~Q\in\mathcal{Q}(D,\,N)^-\end{array}\right.
\quad\textrm{for some}~\gamma\in\Gamma\\
&\Longleftrightarrow\left\{
\begin{array}{ll}\omega_Q=\gamma(\omega_{Q'}) & \textrm{if}~Q\in\mathcal{Q}(D,\,N),\\
\omega_{-Q}=\gamma(\omega_{-Q'})
 & \textrm{if}~Q\in\mathcal{Q}(D,\,N)^-\end{array}\right.
\quad\textrm{for some}~\gamma\in\Gamma\\
&\Longleftrightarrow\left\{
\begin{array}{ll}Q'=Q^\gamma & \textrm{if}~Q\in\mathcal{Q}(D,\,N),\\
-Q'=(-Q)^\gamma=-(Q^\gamma)
 & \textrm{if}~Q\in\mathcal{Q}(D,\,N)^-\end{array}\right.
\quad\textrm{for some}~\gamma\in\Gamma\\
&\Longleftrightarrow Q'=Q^\gamma~\textrm{for some}~\gamma\in\Gamma\\
&\Longleftrightarrow [Q]=[Q']. 
\end{align*} 
This shows that $\rho$ is injective. 
\par
For surjectivity of $\rho$, let $\tau\in\mathbb{H}(D,\,N)^\pm$. 
Let $Q$ be the quadratic form in $\mathcal{Q}(D,\,N)$ having $\tau$ as a zero, and set
\begin{equation*}
Q_\tau=\left\{\begin{array}{rl}
Q & \textrm{if}~\tau\in\mathbb{H},\\
-Q & \textrm{if}~\tau\in\mathbb{H}^-.
\end{array}\right.
\end{equation*}
We then find by the definition (\ref{rho}) that
\begin{equation*}
\rho([Q_\tau])=\left\{\begin{array}{ll}
\rho([Q])=[\omega_Q]=[\tau] & \textrm{if}~\tau\in\mathbb{H},\\
\rho([-Q])= [\overline{\omega}_Q]=[\tau]& \textrm{if}~\tau\in\mathbb{H}^-.
\end{array}\right.
\end{equation*}
Therefore, $\rho$ is surjective and hence bijective. 
\end{proof}

Now, we have
two inverse systems in the category of sets\,:
\begin{enumerate}
\item[(i)] $\{\mathcal{CM}(D,\,Y_1(N)^\pm\}_{N\in\mathbb{N}}$
with the natural covering maps
$Y_1(m)^\pm\twoheadrightarrow Y_1(n)^\pm$ for $n\preccurlyeq m$\,;
\item[(ii)] $\{\mathcal{CM}(D,\,Y(N)^\pm\}_{N\in\mathbb{N}}$
with the natural covering maps
$Y(m)^\pm\twoheadrightarrow Y(n)^\pm$ for $n\preccurlyeq m$.
\end{enumerate}
\begin{align*}
\end{align*}

\begin{definition}\label{CMDH}
Define the following inverse limits\,:
\begin{align*}
\mathcal{CM}(D,\,\widetilde{\mathbb{H}}_1^\pm)&:=
\varprojlim_{N\in\mathbb{N}}\,\mathcal{CM}(D,\,Y_1(N)^\pm)\quad
(\subset\widetilde{\mathbb{H}}^\pm_1=\varprojlim_{N\in\mathbb{N}} Y_1(N)^\pm),\\
\mathcal{CM}(D,\,\widetilde{\mathbb{H}}^\pm)&:=
\varprojlim_{N\in\mathbb{N}}\,\mathcal{CM}(D,\,Y(N)^\pm)\quad
(\subset\widetilde{\mathbb{H}}^\pm=\varprojlim_{N\in\mathbb{N}} Y(N)^\pm).
\end{align*}
\end{definition}

\section {Inverse limits of CM points as Galois groups}

We shall prove our main theorem by interpreting the inverse limits 
\begin{equation*}
\mathcal{CM}(D,\,\widetilde{\mathbb{H}}_1^\pm)\quad\textrm{and}
\quad\mathcal{CM}(D,\,\widetilde{\mathbb{H}}^\pm)
\end{equation*}
in terms of class field theory. 

\begin{lemma}\label{KONinclusion}
Let $K^\mathrm{ab}$ denote the maximal abelian extension 
of the imaginary quadratic field $K=\mathbb{Q}(\sqrt{D})$. 
\begin{enumerate}
\item[\textup{(i)}] If $N,\,M\in\mathbb{N}$ with $N\preccurlyeq M$, then $K_{\mathcal{O},\,N}
\subseteq K_{\mathcal{O},\,M}$. 
\item[\textup{(ii)}] If $\mathcal{O}_1\subseteq\mathcal{O}_2$ are orders in $K$, then 
$K_{\mathcal{O}_1,\,N}\supseteq K_{\mathcal{O}_2,\,N}$. 
\item[\textup{(iii)}] We have $\bigcup_{N\geq1} K_{\mathcal{O},\,N}=K^\mathrm{ab}$. 
\end{enumerate}
\end{lemma}
\begin{proof}
\begin{enumerate}
\item[(i)] See \cite[$\S$4, p. 168]{Stevenhagen}. 
\item[(ii)] See \cite[$\S$4, p. 169]{Stevenhagen}. 
\item[(iii)] We get from (ii) that 
\begin{equation*}
K^\mathrm{ab}\supseteq\bigcup_{N\geq1} K_{\mathcal{O},\,N}
\supseteq\bigcup_{N\geq1} K_{\mathcal{O}_K,\,N}. 
\end{equation*}
The result follows from the classical fact that every finite abelian extension of $K$ is contained in
a ray class field of $K$\,; see \cite[$\S$V.6]{Janusz}. 
\end{enumerate}
\end{proof}

\begin{theorem}\label{main}
Let $D\equiv0$ or $1\Mod{4}$ be a negative integer, 
and let $K=\mathbb{Q}(\sqrt{D})$ be the 
corresponding imaginary quadratic field. 
Let $t$ be a transcendental number. 
\begin{enumerate}
\item[\textup{(i)}] The subset $\mathcal{CM}(D,\,\widetilde{\mathbb{H}}_1^\pm)$ 
of $\widetilde{\mathbb{H}}_1^\pm$ admits a group structure isomorphic to
$\mathrm{Gal}(K^\mathrm{ab}/\mathbb{Q})$.
\item[\textup{(ii)}] 
Define
\begin{equation*}
K^\mathrm{ab}(t^{1/\infty}):=\bigcup_{N\geq1}K^\mathrm{ab}(t^{1/N}). 
\end{equation*}
The subset
$\mathcal{CM}(D,\,\widetilde{\mathbb{H}}^\pm)$ of 
$\widetilde{\mathbb{H}}^\pm$ can be regarded as a group isomorphic to 
$\mathrm{Gal}(K^\mathrm{ab}(t^{1/\infty})/\mathbb{Q}(t))$.
Moreover, $\mathcal{CM}(D,\,\widetilde{\mathbb{H}}_1^\pm)$ 
can be identified with a quotient group
of $\mathcal{CM}(D,\,\widetilde{\mathbb{H}}^\pm)$
 via the natural forgetful maps. 
\end{enumerate}
\end{theorem}
\begin{proof}
\begin{enumerate}
\item[(i)] As sets, we have
\begin{align*}
\mathcal{CM}(D,\,\widetilde{\mathbb{H}}_1^\pm)&=
\varprojlim_{N\in\mathbb{N}}\,\mathcal{CM}(D,\,Y_1(N)^\pm)\\
&\simeq\varprojlim_{N\in\mathbb{N}}\,\mathcal{Q}(D,\,N)^\pm/\sim_{\Gamma_1(N)}\quad\textrm{by Lemma \ref{formpoint}}\\
&\simeq\varprojlim_{N\in\mathbb{N}}\,\mathrm{Gal}(K_{\mathcal{O},\,N}/\mathbb{Q})
\quad\textrm{by Corollary \ref{GalFormCor} (i)}\\
&\simeq\mathrm{Gal}(\bigcup_{N\geq1} K_{\mathcal{O},\,N}/\mathbb{Q})\quad\textrm{by Lemma \ref{KONinclusion} (i)
(cf. \cite[Theorem 164]{I-M} and \cite[$\S$I.2]{Neukirch})}\\
&=\mathrm{Gal}(K^\mathrm{ab}/\mathbb{Q})\quad\textrm{by Lemma \ref{KONinclusion} (iii)}. 
\end{align*}
Thus, $\mathcal{CM}(D,\,\widetilde{\mathbb{H}}_1^\pm)$ inherits a natural group 
structure isomorphic to $\mathrm{Gal}(K^\mathrm{ab}/\mathbb{Q})$.  
\item[(ii)] Similarly, we find that
\begin{align*}
\mathcal{CM}(D,\,\widetilde{\mathbb{H}}^\pm)&=
\varprojlim_{N\in\mathbb{N}}\,\mathcal{CM}(D,\,Y(N)^\pm)\\
&\simeq\varprojlim_{N\in\mathbb{N}}\,\mathcal{Q}(D,\,N)^\pm/\sim_{\Gamma(N)}\quad\textrm{by Lemma \ref{formpoint}}\\
&\simeq\varprojlim_{N\in\mathbb{N}}\,\mathrm{Gal}(K_{\mathcal{O},\,N}(t^{1/N})/\mathbb{Q}(t))\quad\textrm{by Corollary \ref{GalFormCor} (ii)}\\
&\simeq\mathrm{Gal}(\bigcup_{N\geq1} K_{\mathcal{O},\,N}(t^{1/N})/\mathbb{Q}(t))\quad\textrm{by Lemma \ref{KONinclusion} (i)}\\
&=\mathrm{Gal}(K^\mathrm{ab}(t^{1/\infty})/\mathbb{Q}(t))\quad\textrm{by Lemma \ref{KONinclusion} (iii)}.  
\end{align*}
The second assertion follows from (i) together with the natural isomorphism
\begin{equation*}
\mathrm{Gal}(K^\mathrm{ab}/\mathbb{Q})
\simeq
\mathrm{Gal}(K^\mathrm{ab}(t)/\mathbb{Q}(t))\\
\simeq
\mathrm{Gal}(K^\mathrm{ab}(t^{1/\infty})/\mathbb{Q}(t))\,/\,
\mathrm{Gal}(K^\mathrm{ab}(t^{1/\infty})/K^\mathrm{ab}(t)) 
\end{equation*}
and the commutative diagram in Figure \ref{diagram}. 
\end{enumerate}
\end{proof}

\begin{corollary}
Define
\begin{align*}
\mathcal{CM}_1^\pm&:=\varprojlim_{N\in\mathbb{N}}
\{[\tau]\in Y_1(N)^\pm~|~\textrm{$\tau\in\mathbb{H}^\pm$ is a CM point primitive modulo $N$}\}~(\subset\widetilde{\mathbb{H}}_1^\pm),\\
\mathcal{CM}^\pm&:=\varprojlim_{N\in\mathbb{N}}
\{[\tau]\in Y(N)^\pm~|~\textrm{$\tau\in\mathbb{H}^\pm$ is a CM point primitive modulo $N$}\}~(\subset\widetilde{\mathbb{H}}^\pm). 
\end{align*}
Then, as sets, we have the decompositions
\begin{equation*}
\mathcal{CM}_1^\pm\simeq\bigsqcup_{K}\bigsqcup_{\ell=1}^\infty\mathrm{Gal}(K^\mathrm{ab}/\mathbb{Q})
\quad\textrm{and}\quad
\mathcal{CM}^\pm\simeq\bigsqcup_{K}\bigsqcup_{\ell=1}^\infty\mathrm{Gal}(K^\mathrm{ab}(t_{K,\,\ell}^{1/\infty})/\mathbb{Q}(t_{K,\,\ell})),
\end{equation*}
where $K$ ranges over all imaginary quadratic fields, $\ell$ indicates
the conductor of an arbitrary order in $K$, and $\{t_{K,\,\ell}\}_{K,\,\ell}$ is a collection of algebraically independent
transcendental numbers. 
\end{corollary}
\begin{proof}
For each $N\in\mathbb{N}$, we have a partition 
\begin{equation*}
\{[\tau]\in Y_1(N)^\pm~|~\textrm{$\tau\in\mathbb{H}^\pm$ is a CM point primitive modulo $N$}\}
=\bigsqcup_D\mathcal{CM}(D,\,Y_1(N)^\pm),
\end{equation*}
where $D$ runs over all negative integers such that $D\equiv0$ or $1\Mod{4}$.
Taking the inverse limit yields
\begin{align*}
\mathcal{CM}_1^\pm&\simeq\varprojlim_{N\in\mathbb{N}}\bigsqcup_D
\mathcal{CM}(D,\,Y_1(N)^\pm)\\
&\simeq\bigsqcup_D\varprojlim_{N\in\mathbb{N}}
\mathcal{CM}(D,\,Y_1(N)^\pm)\\
&\simeq\bigsqcup_D
\mathrm{Gal}(\mathbb{Q}(\sqrt{D})^\mathrm{ab}/\mathbb{Q})\quad\textrm{by Theorem \ref{main} (i)}\\
&=\bigsqcup_K\bigsqcup_{\ell=1}^\infty
\mathrm{Gal}(\mathbb{Q}(\sqrt{\ell^2d_K})^\mathrm{ab}/\mathbb{Q})\quad\textrm{with}~D=\ell^2d_K\\
&\hspace{2cm}\textrm{where $K$ runs over all imaginary quadratic fields of discriminant $d_K$}\\
&=\bigsqcup_K\bigsqcup_{\ell=1}^\infty
\mathrm{Gal}(K^\mathrm{ab}/\mathbb{Q}).
\end{align*}
\par
The second decomposition follows analogously from Theorem \ref{main} (ii). 
\end{proof}

\section {An explicit description of some Galois groups}

Let $p$ be an odd prime. In this final section, we shall provide
an explicit realization of the Galois group
\begin{equation*} 
\mathrm{Gal}(K_{\mathcal{O},\,p^\infty}(t^{1/p^\infty})/\mathbb{Q}(t))
\end{equation*}
under the assumption $D\neq-3,\,-4$, where $t$ is a transcendental number and
\begin{equation}\label{defKp}
K_{\mathcal{O},\,p^\infty}(t^{1/p^\infty}):=\bigcup_{n\geq1}K_{\mathcal{O},\,p^n}(t^{1/p^n}). 	 
\end{equation}

\begin{lemma}\label{plemma}
Let $\{\gamma_n\}_{n\geq1}$ and $\{\gamma_n'\}_{n\geq1}$ be sequences in 
$\mathrm{SL}_2(\mathbb{Z})$ satisfying\,\textup{:}
\begin{enumerate}
\item[\textup{(i)}] $\gamma_{n+1}\equiv\gamma_n\Mod{p^n M_2(\mathbb{Z})}$ and
$\gamma_{n+1}'\equiv\gamma_n'\Mod{p^n M_2(\mathbb{Z})}$
for all $n\geq1$\,\textup{;}
\item[\textup{(ii)}] $\gamma_1\equiv\gamma_1'\equiv I_2\Mod{p M_2(\mathbb{Z})}$\,\textup{;} 
\item[\textup{(iii)}] For each $n\geq1$, $\gamma_n\equiv\gamma_n'\Mod{p^n M_2(\mathbb{Z})}$ or
$\gamma_n\equiv-\gamma_n'\Mod{p^n M_2(\mathbb{Z})}$. 
\end{enumerate}
Then their $p$-adic limits
$\gamma:=\displaystyle\lim_{n\rightarrow\infty}\gamma_n$ and 
$\gamma':=\displaystyle\lim_{n\rightarrow\infty}\gamma_n'$ in $\mathrm{SL}_2(\mathbb{Z}_p)$ coincide, that is,
$\gamma=\gamma'$. 
\end{lemma}
\begin{proof}
We see from (i) and (ii) that
\begin{equation}\label{allI}
\gamma_n\equiv\gamma_n'\equiv I_2\Mod{pM_2(\mathbb{Z})}\quad
\textrm{for all}~n\geq1. 
\end{equation}
Let $S=\{n\geq1~|~\gamma_n\not\equiv\gamma_n'\Mod{p^n M_2(\mathbb{Z})}\}$, and then $1\not\in S$. 
If $S\neq\emptyset$, then there exists $m\geq2$ such that 
\begin{equation*}
\gamma_m\equiv-\gamma_m'\Mod{p^m M_2(\mathbb{Z})}
\end{equation*}
by (iii). This contradicts (\ref{allI}), since $p\geq3$. 
\par
Thus we conclude that $\gamma_n\equiv\gamma_n'\Mod{p^n M_2(\mathbb{Z})}$ for all $n\geq1$, and hence
$\gamma=\gamma'$. 
\end{proof}

\begin{lemma}\label{isotropy}
Let $\tau\in\mathbb{H}^\pm$ be a CM point. If $D_\tau\neq-3,\,-4$, 
then the isotropy group of $\tau$ in $\mathrm{SL}_2(\mathbb{Z})$ is
$\{I_2,\,-I_2\}$, namely,
$\{\gamma\in\mathrm{SL}_2(\mathbb{Z})~|~\gamma(\tau)=\gamma\}=\{I_2,\,-I_2\}$. 
\end{lemma}
\begin{proof}
See \cite[Proposition 1.5 (c) in Chapter I]{Silverman}. 
\end{proof}

\begin{theorem}
With the same notations as in \textup{Theorem \ref{main}}, let $p$ be an odd prime, and 
let $R\subset\mathbb{H}(D,\,p)^\pm$
be a set of representatives for the elements of $\mathcal{CM}(D,\,Y(p)^\pm)$. 
If $D\neq -3,\,-4$, then there is a one-to-one correspondence between
\begin{equation*}
\mathrm{Gal}(K_{\mathcal{O},\,p^\infty}(t^{1/p^\infty})/\mathbb{Q}(t))\quad\textrm{and}\quad
R\times\{\gamma\in\mathrm{SL}_2(\mathbb{Z}_p)~|~
\gamma\equiv I_2\Mod{pM_2(\mathbb{Z}_p)}\}.
\end{equation*}					
\end{theorem}
\begin{proof}
For $\tau\in\mathbb{H}^\pm$, let $[\tau]_n$ denote its image in $Y(p^n)^\pm$. We derive that 
\begin{align*}
&\mathrm{Gal}(K_{\mathcal{O},\,p^\infty}(t^{1/p^\infty})/\mathbb{Q}(t))\\
=&\mathrm{Gal}(\bigcup_{n\geq1}K_{\mathcal{O},\,p^n}(t^{1/p^n})/\mathbb{Q}(t))
\quad\textrm{by the definition (\ref{defKp}) }\\
\simeq&\varprojlim_{n\geq1}\,\mathrm{Gal}(K_{\mathcal{O},\,p^n}(t^{1/p^n})/\mathbb{Q}(t))
\quad\textrm{by Lemma \ref{KONinclusion} (i)}\\
\simeq&\varprojlim_{n\geq1}\,\mathcal{CM}(D,\,Y(p^n)^\pm)\quad\textrm{by Corollary \ref{GalFormCor} (ii)}\\
\simeq&
\left\{
([\tau]_1,\,[\tau^{\alpha_1}]_2,\,[\tau^{\alpha_1\alpha_2}]_3,\,
[\tau^{\alpha_1\alpha_2\alpha_3}]_4,\,
\ldots)~|~\tau\in\mathbb{H}(D,\,p)^\pm~
\textrm{and}~
\alpha_n\in\Gamma(p^n)~(n\geq1)
\right\}.
\end{align*}
\par
Define a map
\begin{align*}
&\phi:R\times\{\gamma\in\mathrm{SL}_2(\mathbb{Z}_p)~|~
\gamma\equiv I_2\Mod{pM_2(\mathbb{Z}_p)}\}\\&\hspace{1cm}\rightarrow
\left\{
([\tau]_1,\,[\tau^{\alpha_1}]_2,\,[\tau^{\alpha_1\alpha_2}]_3,\,
[\tau^{\alpha_1\alpha_2\alpha_3}]_4,\,\ldots)~|~\tau\in\mathbb{H}(D,\,p)^\pm
~\textrm{and}~
\alpha_n\in\Gamma(p^n)~(n\geq1)
\right\}
\end{align*}
by
\begin{equation*}
\phi(\tau,\,\gamma)=([\tau^{\gamma_1}]_1,\,[\tau^{\gamma_2}]_2,\,[\tau^{\gamma_3}]_3,\,\ldots)
\quad(\tau\in R,~\gamma\in\mathrm{SL}_2(\mathbb{Z}_p)~\textrm{satisfying}~
\gamma\equiv I_2\Mod{pM_2(\mathbb{Z}_p)}),
\end{equation*}
where $\gamma_n\in\mathrm{SL}_2(\mathbb{Z})$ is any lift such that
\begin{equation*}
\gamma_n\equiv\gamma\Mod{p^n M_2(\mathbb{Z}_p)}\quad(n\geq1).
\end{equation*}
It is straightforward to verify that this map 
is a well-defined surjection, independent of the choice of lifts 
$\gamma_n$.
To prove injectivity, suppose that
\begin{equation*}
\phi(\tau,\,\gamma)=\phi(\tau',\,\gamma')\quad
\textrm{for some}~(\tau,\,\gamma),\,(\tau',\,\gamma')
\in R\times\{\gamma\in\mathrm{SL}_2(\mathbb{Z}_p)~|~
\gamma\equiv I_2\Mod{pM_2(\mathbb{Z}_p)}\}.  
\end{equation*}
Then we have
\begin{equation}\label{twoequalimages}
([\tau^{\gamma_1}]_1,\,[\tau^{\gamma_2}]_2,\,[\tau^{\gamma_3}]_3,\,\ldots)=
([\tau'^{\,\gamma_1'}]_1,\,[\tau'^{\,\gamma_2'}]_2,\,[\tau'^{\,\gamma_3'}]_3,\,\ldots)
\end{equation}
where $\gamma_n$ and $\gamma'_n$ are elements of $\mathrm{SL}_2(\mathbb{Z})$ such that
\begin{equation*}
\gamma_n\equiv\gamma\Mod{p^n M_2(\mathbb{Z}_p)}\quad
\textrm{and}\quad\gamma'_n\equiv\gamma'\Mod{p^n M_2(\mathbb{Z}_p)}\quad(n\geq1), 
\end{equation*}
respectively. 
Since $\gamma_1\equiv\gamma\equiv I_2\Mod{pM_2(\mathbb{Z}_p)}$
and $\gamma_1'\equiv\gamma'\equiv I_2\Mod{pM_2(\mathbb{Z}_p)}$, 
we attain $\gamma_1,\,\gamma_1'\in\Gamma(p)$, and so
\begin{equation*}
[\tau]_1=[\tau^{\gamma_1}]_1=
[\tau'^{\,\gamma_1'}]_1=[\tau']_1.
\end{equation*}
It follows by the definition of $R$ that $\tau=\tau'$, and we obtain by (\ref{twoequalimages}) that
\begin{equation}\label{ggb}
\tau^{\gamma_n}=\tau^{\gamma_n'\beta_n}\quad\textrm{for some}~\beta_n\in\Gamma(p^n)\quad
(n\geq1). 
\end{equation}
Since $D_\tau=D\neq-3,\,-4$, the isotropy group of $\tau$
in $\mathrm{SL}_2(\mathbb{Z})$ is $\{I_2,\,-I_2\}$ by Lemma \ref{isotropy}. Thus we deduce by (\ref{ggb}) that
for each $n\geq1$
\begin{equation*}
\gamma_n\equiv\gamma_n'\Mod{p^n M_2(\mathbb{Z})}\quad\textrm{or}
\quad
\gamma_n\equiv-\gamma_n'
\Mod{p^n M_2(\mathbb{Z})}. 
\end{equation*}
Therefore we achieve by Lemma \ref{plemma} that $\gamma=\gamma'$, which proves that $\phi$ is injective and hence bijective.
\end{proof}

\section*{Statements \& Declarations}

\subsection*{Funding}
The first named author was supported by the National Research Foundation of Korea (NRF) grant funded by the Korea government (MSIT) (No. RS-2023-00252986).
The third named (corresponding) author was supported
by Hankuk University of Foreign Studies Research Fund of 2025 and 
by the National Research Foundation of Korea (NRF) grant funded by the Korea government (MSIT) (No. RS-2023-00241953).

\subsection*{Competing interests}
The authors have not disclosed any competing interests.

\subsection*{Data availability}
Data sharing is not applicable to this article as no datasets
were generated or analysed during the current study.

\bibliographystyle{amsplain}

\address{
Department of Mathematics\\
Dankook University\\
Cheonan-si, Chungnam 31116\\
Republic of Korea} {hoyunjung@dankook.ac.kr}

\address{
Department of Mathematical Sciences \\
KAIST \\
Daejeon 34141\\
Republic of Korea} {jkgoo@kaist.ac.kr}

\address{
Department of Mathematics\\
Hankuk University of Foreign Studies\\
Yongin-si, Gyeonggi-do 17035\\
Republic of Korea} {dhshin@hufs.ac.kr}

\end{document}